\documentclass[12pt,a4paper,UKenglish]{article}
\usepackage[margin=3cm]{geometry}
\usepackage{amsthm} 
\usepackage{amssymb}
\usepackage{amsopn}
\usepackage{color}
\usepackage{babel} 
\usepackage{hyperref}
\usepackage{amsmath, amsfonts, amssymb} 
\usepackage{graphicx}
\allowdisplaybreaks 
\usepackage[utf8]{inputenc}
\usepackage[numbers,sort&compress]{natbib}

\usepackage{esvect}

\usepackage{tikz}
\usetikzlibrary{topaths,calc}

\usepackage{amssymb, amsmath}
\usepackage{verbatim}
\usepackage{hyperref}

\theoremstyle{plain}
\newtheorem{theorem}{Theorem}[section]
\newtheorem{lemma}[theorem]{Lemma}
\newtheorem{corollary}[theorem]{Corollary}
\newtheorem{proposition}[theorem]{Proposition}

\theoremstyle{definition}
\newtheorem{definition}[theorem]{Definition}

\theoremstyle{remark}

\newtheorem{remark}[theorem]{Remark}

\DeclareMathOperator{\RQ}{RQ}
\DeclareMathOperator{\vol}{vol}

\usepackage[affil-it]{authblk}
\begin{document}

\title{Spectra of Complex Unit Hypergraphs}

\author[1,2]{Raffaella Mulas\footnote{Email address: r.mulas@vu.nl}}
\author[3]{Nathan Reff}%\footnote{Email address: nathan.reff@gmail.com}}
\affil[1]{Vrije Universiteit Amsterdam, Amsterdam, The Netherlands}
\affil[2]{Max Planck Institute for Mathematics in the Sciences, Leipzig, Germany}
\affil[3]{Relativity, 231 South LaSalle Street, 8th Floor. Chicago, IL 60604, USA}

\date{}
	
\maketitle

\begin{abstract}
A complex unit hypergraph is a hypergraph where each vertex-edge incidence is given a complex unit label.  We define the adjacency, incidence, Kirchoff Laplacian and normalized Laplacian of a complex unit hypergraph and study each of them.  Eigenvalue bounds for the adjacency, Kirchoff Laplacian and normalized Laplacian are also found.
Complex unit hypergraphs naturally generalize several hypergraphic structures such as oriented hypergraphs, where vertex-edge incidences are labelled as either $+1$ or $-1$, as well as ordinary hypergraphs.  Complex unit hypergraphs also generalize their graphic analogues, which are complex unit gain graphs, signed graphs, and ordinary graphs.\\
%Oriented hypergraphs are generalized to the case where each vertex--edge incidence has a coefficient from the complex unit circle. Their associated adjacency matrix, Kirchhoff Laplacian and normalized Laplacian are introduced and investigated.\newline
\vspace{0.2cm}
\noindent {\bf Keywords:} Hypergraphs, Laplace operators, Adjacency matrix, Spectrum
\end{abstract}

	\section{Introduction}
	Spectral theory of graphs studies the eigenvalues of the adjacency matrix, the Kirchhoff Laplacian and the normalized Laplacian associated to a graph \cite{Chung,BH2012}. Such eigenvalues are known to identify many, if not most, important qualitative properties of a given graph, and they can be easily computed with tools from linear algebra. For these reasons, spectral graph theory finds application in many disciplines and it has been widely investigated.
	
	As an extension of this theory, graph operators have been introduced and studied for \emph{hypergraphs}; a generalization of graphs in which edges do not necessarily join only \emph{pairs} of vertices but rather \emph{sets of vertices} of any cardinality. This allows us to model communities of elements of any size, for instance, chemical reactions involving \emph{sets} of chemical elements or research articles whose authors are \emph{groups} of people and not necessarily pairs. Hypergraphs are therefore very interesting objects both from the mathematical point of view and due to their applicability in network science. The study of their spectra is an active field of research. Moreover, further generalizations of the theory include the existence of a given coefficient for each vertex--edge incidence in a hypergraph: \emph{oriented hypergraphs} \cite{ReffRusnak} have coefficients in $\mathbb{Z}\setminus \{0\}$; \emph{chemical hypergraphs} \cite{JM2019} have coefficients in $\{-1,0,+1\}$; \emph{hypergraphs with real coefficients} \cite{JM2020} have coefficients in $\mathbb{R}\setminus \{0\}$. The spectra of the Laplace operators and the adjacency matrix have been studied for oriented hypergraphs, while the spectrum of the normalized Laplacian has been investigated for chemical hypergraphs and hypergraphs with real coefficients.  We refer the reader to \cite{hyp2013,hyp2014,hyp2016,hyp2019,hyp2019-3,Sharp,Independence,Symmetries,Rusnak21} for a significant, if incomplete, selection of literature on this topic. Moreover, we refer to \cite{yu2019,wang2022,Galuppi21} for a selection of literature on the spectral theory of signed hypergraphs and weighted hypergraphs in which higher dimensional tensors are considered. 
	
	Here we introduce a generalization of oriented hypergraphs in which the coefficient of a vertex--edge incidence is an element of the complex unit circle. We call them \emph{complex unit hypergraphs}. We also define their associated adjacency, Kirchhoff Laplacian and normalized Laplacian matrices, as operators that have entries in the complex field.
	
	The paper is structured as follows. In Section \ref{Section Basic definitions}, we give the basic definitions on complex unit hypergraphs and their associated operators. In Section \ref{Section First properties}, we investigate the first properties of the spectra and in Section \ref{Section Hypergraph transformations} we discuss hypergraph transformations and their effect on the eigenvalues. Finally, in Section \ref{Section: smallest and largest}, we provide several bounds for the smallest and largest eigenvalues of each operator.

\section{Basic definitions}\label{Section Basic definitions}

In this section we are going to present the basic definitions that will be needed throughout the paper. We note that different references might use different definitions, and we do not claim that the definitions that are presented here should be the standard ones.

	\begin{definition}
			A \emph{hypergraph} is a triple $(V,E,\mathcal{I})$ such that:
			\begin{itemize}
				\item $V=\{v_1,\ldots,v_n\}$ is a finite set of \emph{nodes} or \emph{vertices};
				\item $E=\{e_1,\ldots,e_m\}$ is a finite set of \emph{edges};
				\item $\mathcal{I}\subseteq V\times E$ is a set of \emph{incidences}.
					\end{itemize}
					If $(v,e)\in \mathcal{I}$, $v$ and $e$ are \emph{incident} and we denote it by $v\in e$. If $v_i\neq v_j$ are both incident to a given edge $e$, then $v_i$ and $v_j$ are \emph{adjacent}, denoted $v_i\sim v_j$, and $e$ \emph{joins} $v_i$ and $v_j$. The \emph{set of oriented adjacencies} is
					\begin{equation*}
					    \vv{\mathcal{A}}:=\{(e,v_i,v_j)\in E\times V\times V: e \text{ joins }v_i\text{ and }v_j \}.
					\end{equation*}
		\end{definition}

\begin{remark}
The above definition of adjacent vertices does not include self-loops. Therefore, it differs from the definition in \cite{ReffRusnak,hyp2018-2}.
\end{remark}
		\begin{definition}The degree of a vertex $v_j$, denoted by $d_j = deg(v_j)$, is equal to the number of incidences containing $v_j$. The size of an edge $e$ is the number of incidences which are contained in
$e$. A $k$-edge is an edge of size $k$. A $k$-uniform hypergraph is a hypergraph
such that all of its edges have size $k$. A $d$-regular hypergraph is a hypergraph
where ever vertex has degree $d$.
		\end{definition}
		We let $\mathbb{T}$ denote the multiplicative group of complex units.
		
	\begin{definition}A \emph{complex unit hypergraph} is a quadruple $G=(V,E,\mathcal{I},\omega)$ consisting of a hypergraph $(V,E,\mathcal{I})$ and an {\it incidence phase function} $\omega:V\times E\rightarrow\mathbb{T}\cup\{0\}$ that satisfies
		\begin{equation*}
				  \omega(v,e)\neq 0 \iff  v\in e.
				\end{equation*}	    
					\end{definition}

\begin{remark}
Here we choose to take $V\times E$ as the domain of the incidence phase function $\omega$, and to let $\omega$ be zero outside of $\mathcal{I}$. Alternatively, as done for instance in \cite{ReffRusnak}, one can choose to take $\mathcal{I}$ as the domain in the definition of  incidence phase function.
\end{remark}

				From here on we fix a complex unit hypergraph $G=(V,E,\mathcal{I},\omega)$. Moreover, we let
		\begin{equation*}
		    \varphi:\mathcal{\vv{\mathcal{A}}}\rightarrow \mathbb{T},
		\end{equation*}called the {\it adjacency gain function}, be defined by
		\begin{equation*}
		   \varphi_e(v_i,v_j):=\varphi(e,v_i,v_j)=- \omega(v_i,e)\cdot \omega(v_j,e)^{-1}.
		\end{equation*}
	\begin{remark}\label{rmkphi}For all $e\in E$ and for all $v_i\neq v_j\in e$,
	\begin{align*}
	    \varphi_e(v_i,v_j)&=-\omega(v_i,e)\cdot \omega(v_j,e)^{-1}\\
	    &=[-\omega(v_j,e)\cdot \omega(v_i,e)^{-1}]^{-1}\\
	    &=\varphi_e(v_j,v_i)^{-1}.
	\end{align*}
	    	\end{remark}
	    	
%\noindent {\bf Example:} See Figure \ref{CHExample} for an example of a complex unit hypergraph $G=(V,E,\mathcal{I},\omega)$.  	    	

%\begin{figure}[h!]
%    \includegraphics[scale=0.75]{CUhypergraphExample.pdf}\centering
%    \caption{A complex unit hypergraph $G$.  Edge $e_1$ is a 4-edge, $e_2$ is an 2-edge, and edge $e_3$ has 3-edge.  The incidence labels (incidence phase function values) are the $\alpha_i$'s, $\beta_i$'s and $\gamma_i$'s colored in blue, each of which is an element of $\mathbb{T}$.  Here adjacency gain values for the two oriented adjacencies with $e_2$ are shown and colored in red.  To make this picture much simpler, the other adjacency gain values are left out. }\label{CHExample}
%  \end{figure}

\noindent {\bf Example:} See Figure \ref{CHExample} for an example of a complex unit hypergraph $G=(V,E,\mathcal{I},\omega)$.  	    	

\begin{figure}[h!]
    \includegraphics[scale=0.75]{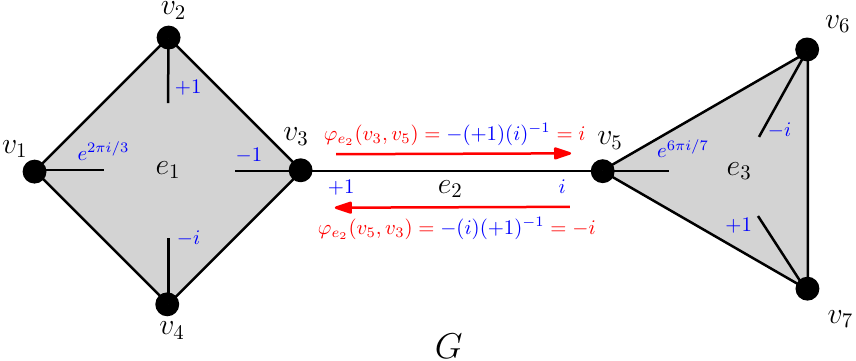}\centering
    \caption{A complex unit hypergraph $G$.  Edge $e_1$ is a 4-edge, $e_2$ is an 2-edge, and edge $e_3$ has 3-edge.  The incidence labels (incidence phase function values) are colored in blue.  Here adjacency gain values for the two oriented adjacencies with $e_2$ are shown and colored in red.  To make this picture much simpler, the other adjacency gain values are left out. }\label{CHExample}
\end{figure}

	    	\begin{remark} A 2-uniform complex unit hypergraph is a \emph{$\mathbb{T}$-oriented gain graph} \cite{MR3530682}.  If one ignores the incidence phase function, but preserves the adjacency gain function values this is a \emph{complex unit gain graph} (or \emph{$\mathbb{T}$-gain graph}).  Complex unit hypergraphs are the natural hypergraph analogue of these types of gain graphs, of which this paper generalizes much of their spectral properties \cite{MR2900705}.
	    	\end{remark}
	    	
	    	\begin{remark}An \emph{oriented hypergraph} \cite{ReffRusnak} is a complex unit hypergraph such that $\omega:V\times E\rightarrow\{-1,0,+1\}$.  A 2-uniform oriented hypergraph is an {\it oriented signed graph} \cite{MR1120422}, which also generalizes bidirected graphs \cite{MR0267898}.  If one ignores the incidence labels, but preserves the adjacency signs, this is a \emph{signed graph}.  A \emph{signed simple graph} is an oriented hypergraph such that:
		   \begin{itemize}
		       \item[-] $E$ is a set (that is, $j\neq k$ implies $e_j\neq e_k$);
		       \item[-] Each edge contains exactly two vertices.
		   \end{itemize}A \emph{simple graph} is a signed graph such that, for each edge $e$, there exists a unique $v\in e$ such that $\omega(v,e)=1$ and there exists a unique $w\in e$ such that $\omega(w,e)=-1$.
	    	\end{remark}
	    	
We now define the operators to $G$. Given a complex matrix $M$, we denote by $M^+$ its conjugate transpose.

	    	\begin{definition}The \emph{degree matrix} of complex unit hypergraph $G$ is
\begin{equation*}
    D:=D(G)=\textrm{diag}(\deg v_1,\ldots,\deg v_n).
\end{equation*}

The \emph{incidence matrix} of $G$ is $B:=B(G)=(B_{ij})\in (\mathbb{T}\cup\{0\})^{n\times m}$, where
 \begin{equation*}
        B_{ij}:=\begin{cases}\omega(v_i,e_j) &\text{ if }v_i\in e_j,\\ 0 &\text{otherwise}.\end{cases}
    \end{equation*}
The \emph{adjacency matrix} of $G$ is $A:=A(G)=(a_{ij})\in \mathbb{C}^{n\times n}$, where
    \begin{equation*}
        a_{ij}:=\begin{cases}\sum_{e\in E}\varphi_e(v_i,v_j) &\text{ if }v_i\sim v_j,\\ 0 &\text{otherwise}.\end{cases}
    \end{equation*}
The \emph{Kirchhoff Laplacian} of $G$ is the $n\times n$ matrix
\begin{equation*}
    K:=K(G)=D-A;
\end{equation*}
The \emph{dual Kirchhoff Laplacian} of $G$ is the $m\times m$ matrix
\begin{equation*}
    K^*:=K^*(G)=B^+B;
\end{equation*}
The \emph{normalized Laplacian} of $G$ is the $n\times n$ matrix
\begin{equation*}
    L:=L(G)=D^{-1}K=\textrm{Id}-D^{-1}A;
\end{equation*}
The \emph{dual normalized Laplacian} of $G$ is the $m\times m$ matrix
\begin{equation*}
    L^*:=L^*(G)=B^+D^{-1}B.
\end{equation*}
    \end{definition}
    \begin{definition}
Two hypergraphs $G_1$ and $G_2$ are \emph{cospectral with respect to} a given operator $M$ if $M(G_1)$ and $M(G_2)$ have the same spectrum.
    \end{definition}
\section{First properties}\label{Section First properties}
\begin{remark}\label{rmk:Chung}
    If $v_i$ and $v_j$ are adjacent, then, by Remark \ref{rmkphi},
\begin{equation*}
    a_{ij}=\sum_{e\in E}\varphi_e(v_i,v_j)=\sum_{e\in E}\varphi_e(v_j,v_i)^{-1}=\sum_{e\in E}\overline{\varphi_e(v_j,v_i)}=\overline{a_{ji}}.
\end{equation*}Therefore, $A$ and $K$ are Hermitian matrices and, in particular, they have real eigenvalues. Moreover, the normalized Laplacian $L$ is similar to the Hermitian matrix
    \begin{equation*}
       \mathcal{L}:=\mathcal{L}(G)= D^{1/2}LD^{-1/2}=\textrm{Id}-D^{-1/2}AD^{-1/2},
    \end{equation*}hence $L$ and $\mathcal{L}$ share the same (real) eigenvalues. Also, $\mathbf{x}$ is an eigenvector for $\mathcal{L}$ with eigenvalue $\lambda$ if and only if $D^{-1/2}\mathbf{x}$ is an eigenvector for $L$ with eigenvalue $\lambda$.
\end{remark}
\begin{remark}
    If $G$ is $d$-regular,
\begin{align*}
    \lambda \text{ is an eigenvalue for }K &\iff \frac{\lambda}{d} \text{ is an eigenvalue for }L\\
    &\iff d-\lambda \text{ is an eigenvalue for }A.
\end{align*}Therefore, for regular complex unit hypergraphs, the spectra of $K$, $L$ and $A$ are all equivalent, up to an additive or multiplicative constant.
\end{remark}
    \begin{theorem}\label{thm:Laplacians}
    The Kirchhoff Laplacian and the normalized Laplacian can be rewritten as
    \begin{equation*}
        K=BB^+ \quad \text{and}\quad L=D^{-1}BB^+,
    \end{equation*}
    respectively.
    \end{theorem}
    \begin{proof}Observe that
    \begin{equation*}
        (BB^+)_{ij}=\sum_{k=1}^m \omega(v_i,e_k)\overline{\omega(v_j,e_k)}.
    \end{equation*}If $i=j$, then the sum simplifies to 
  \begin{equation*}
      \sum_{k=1}^m |\omega(v_i,e_k)|^2=\deg v_i,
  \end{equation*}since $|\omega(v_i,e_k)|=1$ if $v_i\in e_k$. If $i\neq j$, then the sum simplifies to 
  \begin{equation*}
      \sum_{k=1}^m \omega(v_i,e_k)\overline{\omega(v_j,e_k)}=\sum_{k=1}^m-\varphi_{e_k}(v_i,v_j)=-a_{ij}.
  \end{equation*}Therefore, $K=D-A=BB^+$ and $L=D^{-1}K=D^{-1}BB^+$.
    \end{proof}
    \begin{corollary}\label{cor:E}
    $K$ and $K^*$ have the same non-zero eigenvalues.    Similarly, $L$ and $L^*$ have the same non-zero eigenvalues.
    \end{corollary}
    \begin{proof}It follows from the fact that, if $f$ and $g$ are linear operators, then the non-zero eigenvalues of $fg$ and $gf$ are the same.
    \end{proof}
\begin{corollary}\label{cor:dual}
Given a matrix $M$, let $\mu_0(M)$ denote the multiplicity of the eigenvalue $0$ for $M$. We have that
\begin{equation}\label{eq:n-m}
    \mu_0(K)-\mu_0(K^*)=\mu_0(L)-\mu_0(L^*)=n-m,
\end{equation}
\begin{equation}\label{eq:mu01}
    \mu_0(K)=\mu_0(L) \quad\text{and}\quad  \mu_0(K^*)=\mu_0(L^*).
\end{equation}
\end{corollary} 
\begin{proof}
\eqref{eq:n-m} is an immediate consequence of Corollary \ref{cor:E}. \eqref{eq:mu01} follows from the fact that $K=BB^+$ while $L=D^{-1}BB^+$, and from Corollary \ref{cor:E}. 
\end{proof}
    \begin{proposition}\label{prop:nonnegative}The eigenvalues of $K$ and $L$ are non-negative.
    \end{proposition}
    \begin{proof}
It follows from Theorem \ref{thm:Laplacians} and Corollary \ref{cor:E}, since $K=BB^+$ and $L$ has the same non-zero eigenvalues as
\begin{equation*}
    L^*=B^+D^{-1}B=B^+D^{-1/2}D^{-1/2}B=(D^{-1/2}B)^+D^{-1/2}B. \qedhere
\end{equation*}
    \end{proof}
    
    \begin{proposition}
        \begin{equation*}
        \ker(K)=\ker(L)=\ker(B^+) \quad \text{and}\quad \ker(K^+)=\ker(L^+)=\ker(B).
    \end{equation*}
        \end{proposition}

    \begin{proof}
  The equality $\ker(K)=\ker(L)$ follows from the fact that $K=BB^+$ while $L=D^{-1}BB^+$. Now, given $\mathbf{x}\in\mathbb{C}^n$,
  \begin{equation*}
      \mathbf{x}\in \ker(K) \iff \mathbf{x}^+K\mathbf{x}= \mathbf{x}^+BB^+\mathbf{x}=(B^+\mathbf{x})^+B^+\mathbf{x}=0\iff \mathbf{x}\in \ker(B^+).
  \end{equation*}This proves the first claim. The second one is analogous.
    \end{proof}
    
Given an $n\times n$ matrix $M$ with real eigenvalues, we will denote its spectrum by
\begin{equation*}
    \lambda_1(M)\leq \ldots\leq \lambda_n(M).
\end{equation*}
\begin{remark}\label{rmk:trace}
    Since the trace of a matrix equals the sum of its eigenvalues,
  \begin{itemize}
      \item $\sum_{i=1}^n\lambda_i(A)=0$;
      \item $\sum_{i=1}^n\lambda_i(K)=\sum_{j=1}^m\lambda_j(K^*)=\sum_{v\in V}\deg v=\sum_{e\in E}|e|$;
      \item $\sum_{i=1}^n\lambda_i(L)=\sum_{j=1}^m\lambda_j(L^\top)=n$.
  \end{itemize}
\end{remark}

\section{Hypergraph transformations}\label{Section Hypergraph transformations}
In this section we discuss some hypergraph transformations and their effect on the spectra.

\subsection{Duality}
\begin{definition}Given $G=(V,E,\mathcal{I},\omega)$, its \emph{dual hypergraph} is $G^*:=(E,V,\mathcal{I}^*,\omega^*)$, where 
\begin{equation*}
    \mathcal{I}^*:=\{(e,v):(v,e)\in\mathcal{I}\},
\end{equation*}and
$\omega^*: E\times V\rightarrow\mathbb{T}\cup\{0\}$ is defined by
\begin{equation*}
    \omega^*(e,v):=\omega(v,e)^{-1}.
\end{equation*}\end{definition}

\begin{remark}\label{rmk:dual}
Clearly,
\begin{itemize}
    \item The degree of a vertex in $G$ equals the size of the corresponding edge in $G^*$ and the size of an edge in $G$ equals the degree of the corresponding vertex in $G^*$;
    \item $(G^*)^*=G$;
    \item $B(G^*)=B(G)^+$;
    \item $K(G^*)=B^+B=K^*(G)$. 
\end{itemize}
In particular, in view of Corollary \ref{cor:dual}, $G$ and its dual hypergraph have the same non-zero eigenvalues with respect to the Kirchhoff Laplacian $K$. The same doesn't hold, in general, for the normalized Laplacian.
\end{remark}
 \begin{proposition}If $G$ is $d$-regular and $m$-uniform, then
 \begin{equation*}
     L(G^*)=\frac{d}{m}\cdot L^*(G)\quad \text{and}\quad L^*(G^*)=\frac{d}{m}\cdot L(G).
 \end{equation*}
    \end{proposition}
\begin{proof}
Since $G$ is $d$-regular and $m$-uniform, $G^*$ is $m$-regular and $d$-uniform. Hence,
\begin{equation*}
    L(G)=\frac{1}{d}\cdot \mathrm{Id}\cdot BB^+ \quad \text{and}\quad 
    L^*(G)=B^+\cdot \frac{1}{d}\cdot \textrm{Id}\cdot B.
\end{equation*}while
\begin{equation*}
     L(G^*)=\frac{1}{m}\cdot \mathrm{Id}\cdot B^+B=\frac{d}{m}\cdot L^*(G)
\end{equation*}and similarly $L^*(G^*)=\frac{d}{m}\cdot L(G)$.
\end{proof}
\subsection{Weak vertex deletion and weak edge deletion}

We shall now consider weak vertex deletion and weak edge deletion \cite{hyp2013}.
\begin{definition}Let $G=(V,E,\mathcal{I},\omega)$.\newline
Given $v\in V$, we let $G\setminus v:=(V\setminus \{v\},E,\mathcal{I}_v,\omega_v)$, where:
\begin{itemize}
 %   \item $E_v:=\{e\setminus \{v\}:e\in E\}$,
    \item $\mathcal{I}_v:=\mathcal{I}\cap \bigl((V\setminus\{v\})\times E\bigr)$, and
    \item $\omega_v:=\omega|_{(V\setminus\{v\})\times E}$ is $\omega$ restricted to $(V\setminus\{v\})\times E$.
\end{itemize}We say that $G\setminus v$ is obtained from $G$ by a \emph{weak vertex deletion} of $v$.
\end{definition}
We will apply the Cauchy Interlacing Theorem \cite[Theorem 4.3.17]{MatrixAnalysis} in order to prove the results in this subsection.

\begin{theorem}[Cauchy Interlacing Theorem]\label{thm:Cauchy}
Let $M\in \mathbb{C}^{n\times n}$ be Hermitian, let $r\in \{1,\ldots,n-1\}$ and let $M_r$ be an $r\times r$ principal submatrix of $M$. Then 
\begin{equation*}
        \lambda_{k}(M)\leq \lambda_k(M_r)\leq \lambda_{k+r}(M)\qquad\text{for all }k\in\{1,\ldots,n-r\}.
    \end{equation*}
\end{theorem}

\begin{theorem}
Let $M$ be any of the operators $A$, $K$ or $L$. If $\hat{G}$ is obtained from $G$ by weak-deleting $r$ vertices,
\begin{equation*}
        \lambda_{k}(M(G))\leq \lambda_k(M(\hat{G}))\leq \lambda_{k+r}(M(G))\qquad\text{for all }k\in\{1,\ldots,n-r\}.
    \end{equation*}
\end{theorem}
\begin{proof}
Notice that $A(\hat{G})$ and $K(\hat{G})$ are obtained from $A(G)$ and $K(G)$, respectively, by removing the $r$ rows and columns corresponding to the deleted vertices. By Theorem \ref{thm:Cauchy}, this proves the claim for $A$ and $K$.\newline

In order to prove the claim for $L$, recall from Remark \ref{rmk:Chung} that $L$ is cospectral to the Hermitian matrix
    \begin{equation*}
       \mathcal{L}=\textrm{Id}-D^{-1/2}AD^{-1/2},
    \end{equation*}therefore it suffices to prove the claim for $\mathcal{L}$. Since $\mathcal{L}(\hat{G})$ is obtained from $\mathcal{L}(G)$ by removing the $r$ rows and columns corresponding to the deleted vertices, by Theorem \ref{thm:Cauchy} this proves the claim.
\end{proof}

\begin{definition}
Given $e\in E$, we let $G\setminus e:=(V,E\setminus\{e\},\mathcal{I}_e,\omega_e)$, where:
\begin{itemize}
    \item $\mathcal{I}_v:=\mathcal{I}\cap \bigl(V\times (E\setminus\{e\})\bigr)$, and
    \item $\omega_e:=\omega|_{V\times (E\setminus\{e\})}$.
\end{itemize}We say that $G\setminus e$ is obtained from $G$ by a \emph{weak edge deletion} of $e$. We say that $G$ is obtained from $G\setminus e$ by a \emph{weak edge addition} of $e$.
\end{definition}
\begin{theorem}
If $\hat{G}$ is obtained from $G$ by weak-deleting $r$ edges,
\begin{equation*}
        \lambda_{j-r}(K(G))\leq \lambda_j(K(\hat{G}))\leq \lambda_j(K(G))\qquad\text{for all }j\in\{r+1,\ldots,n\}.
    \end{equation*}\end{theorem}
    \begin{proof}
    Observe that $B(\hat{G})$ is obtained from $B(G)$ by deleting the columns of $B(G)$ corresponding to the $r$ deleted edges. Therefore, $K^*(\hat{G})=B(\hat{G})^+B(\hat{G})$ is an $r\times r$ principal submatrix of $K^*(G)=B(G)^+B(G)$. By Corollary \ref{cor:E} together with Theorem \ref{thm:Cauchy}, this proves the claim.
    \end{proof}
\begin{remark}While the Cauchy Interlacing Theorem can be applied to $A$, $K$ and $L$ in the case of weak vertex deletion, it can only be applied to $K$ in the case of weak edge deletion. Interestingly, in the case of simple graphs there are some known interlacing results for $L$ in the case of edge deletion \cite{interlacing1,interlacing2,interlacing3}. Such results don't make use of the Cauchy Interlacing Theorem and generalizing them to the case of hypergraphs remains an open problem.
\end{remark}

\subsection{Vertex and edge switching}
\begin{definition}A \emph{vertex switching function} is any function $\zeta:V\rightarrow \mathbb{T}$. \emph{Vertex-switching} the complex unit hypergraph $G=(V,E,\mathcal{I},\omega)$ means replacing $\omega$ with $\omega^{\zeta}$, defined by
\begin{equation*}
    \omega^\zeta(v,e)=\zeta(v)^{-1}\omega(v,e);
\end{equation*}producing the complex unit hypergraph $G^\zeta=(V,E,\mathcal{I},\omega^\zeta)$.\newline

Similarly, an \emph{edge switching function} is any function $\xi:E\rightarrow \mathbb{T}$. \emph{Edge-switching} the complex unit hypergraph $G=(V,E,\mathcal{I},\omega)$ means replacing $\omega$ with $\omega^{\xi}$, defined by
\begin{equation*}
    \omega^\xi(v,e)=\xi(e)^{-1}\omega(v,e);
\end{equation*}producing the complex unit hypergraph $G^\xi=(V,E,\mathcal{I},\omega^\xi)$.

\end{definition}
\begin{remark}Vertex-switching and edge-switching are equivalence relations.\end{remark}
\begin{definition}
For a vertex-switching function $\zeta$, we let
\begin{equation*}
    D_n(\zeta):=\textrm{diag}(\zeta(v_1),\ldots,\zeta(v_n)).
\end{equation*}For an edge-switching function $\xi$, we let
\begin{equation*}
    D_m(\xi):=\textrm{diag}(\xi(e_1),\ldots,\xi(e_m)).
\end{equation*}
\end{definition}
\begin{lemma}\label{lemma:switching}
If $\zeta$ is a vertex-switching function on $G$,
\begin{itemize}
    \item $B(G^\zeta)=D_n(\zeta)^+B(G)$;
    \item $A(G^\zeta)=D_n(\zeta)^+A(G)D_n(\zeta)$; 
    \item $K(G^\zeta)=D_n(\zeta)^+K(G)D_n(\zeta)$;
    \item $L(G^\zeta)=D_n(\zeta)^+L(G)D_n(\zeta)$.
\end{itemize}Moreover, if $\xi$ is and edge-switching function on $G$,
\begin{itemize}
    \item $B(G^\xi)=B(G)D_m(\xi)$;
    \item $A(G^\xi)=A(G)$; 
    \item $K(G^\xi)=K(G)$;
    \item $L(G^\xi)=L(G)$.
\end{itemize}
\end{lemma}
\begin{corollary}
For any vertex-switching function $\zeta$, $G$ and $G^\zeta$ are cospectral with respect to $A$, $K$, $K^*$, $L$ and $L^*$.
\end{corollary}

\section{Smallest and largest eigenvalue}\label{Section: smallest and largest}
In this section we apply the Courant–Fischer–Weyl min-max principle \cite[Theorem 4.2.2]{MatrixAnalysis} and other preliminary lemmas in order to
estimate the smallest and the largest eigenvalue of $A$, $K$ and $L$, respectively.

\subsection{Preliminary results}
\begin{lemma}[Courant–Fischer–Weyl min-max principle]\label{minmax}
Let $M\in \mathbb{C}^{n\times n}$ be Hermitian. Then
\begin{equation*}
    \lambda_1(M)=\min_{\mathbf{x}\in \mathbb{C}^{n}\setminus\{0\}}\frac{\mathbf{x}^+M\mathbf{x}}{\mathbf{x}^+\mathbf{x}}=\min_{\mathbf{x}^+\mathbf{x}=1}\mathbf{x}^+M\mathbf{x},
\end{equation*}and
\begin{equation*}
    \lambda_n(M)=\max_{\mathbf{x}\in \mathbb{C}^{n}\setminus\{0\}}\frac{\mathbf{x}^+M\mathbf{x}}{\mathbf{x}^+\mathbf{x}}=\max_{\mathbf{x}^+\mathbf{x}=1}\mathbf{x}^+M\mathbf{x}.
\end{equation*}The vectors realizing such min or max are then are corresponding eigenvectors.
\end{lemma}

\begin{proposition}\label{prop:xKx}Let $\mathbf{x}=(x_1,\ldots,x_n)\in\mathbb{C}^n$. Then
\begin{equation*}
    \mathbf{x}^+K\mathbf{x}=\sum_{e\in E}\bigl|\sum_{v_j\in e}\omega(v_j,e)^{-1}x_j\bigr|^2,
\end{equation*}and
\begin{equation*}
    \mathbf{x}^+\mathcal{L}\mathbf{x}=\sum_{e\in E}\biggl|\sum_{v_j\in e}\frac{\omega(v_j,e)^{-1}}{{\sqrt{deg(v_i)}}}x_j\biggr|^2.
\end{equation*}
\end{proposition}
\begin{proof}Observe that, given $\mathbf{x}=(x_1,\ldots,x_n)\in\mathbb{C}^n$.
\begin{equation*}
    (\mathbf{x}^+B)^+=B^+\mathbf{x}=\biggl(\sum_{j=1}^n\overline{\omega(v_j,e_1)}x_j,\ldots,\sum_{j=1}^n\overline{\omega(v_j,e_m)}x_j\biggr).
\end{equation*}Therefore,
\begin{align*}
    \mathbf{x}^+K\mathbf{x}&=\mathbf{x}^+BB^+\mathbf{x}\\
    &=(\mathbf{x}^+B)(\mathbf{x}^+B)^+\\
    &=\sum_{k=1}^m\bigl|\sum_{j=1}^n\overline{\omega(v_j,e_k)}x_j\bigr|^2\\
    &=\sum_{k=1}^m\bigl|\sum_{j=1}^n\omega(v_j,e_k)^{-1}x_j\bigr|^2.
\end{align*}Similarly, since $\mathcal{L}=D^{1/2}LD^{-1/2}=D^{-1/2}BB^+D^{-1/2}$,
\begin{align*}
    \mathbf{x}^+\mathcal{L}\mathbf{x}&=\mathbf{x}^+D^{-1/2}BB^+D^{-1/2}\mathbf{x}\\
    &=(\mathbf{x}^+D^{-1/2}B)(\mathbf{x}^+D^{-1/2}B)^+\\
    &=\sum_{k=1}^m\biggl|\sum_{j=1}^n\frac{\omega(v_j,e_k)^{-1}}{\sqrt{\deg v_j}}x_j\biggr|^2.
    \end{align*}
\end{proof}

\begin{corollary}\label{cor:rq}
The smallest (largest) eigenvalue of $A$, $K$ and $L$, respectively, is the minimizer (maximizer) of the Rayleigh quotient
\begin{equation*}
\RQ_A(\mathbf{x}):=\frac{\sum_{v_i\sim v_j} \sum_{e\in E}\overline{x_i}\cdot \varphi_e(v_i,v_j)\cdot x_j}{\sum_{i=1}^n \overline{x_i}\cdot x_i},
\end{equation*}
\begin{equation*}
    \RQ_K(\mathbf{x}):=\frac{\sum_{e\in E}\bigl|\sum_{v_j\in e}\omega(v_j,e)^{-1}x_j\bigr|^2}{\sum_{i=1}^n \overline{x_i}\cdot x_i}.
\end{equation*}and
\begin{equation*}
    \RQ_L(\mathbf{x}):=\frac{\sum_{e\in E}\bigl|\sum_{v_j\in e}\omega(v_j,e)^{-1}x_j\bigr|^2}{\sum_{i=1}^n \deg v_i\cdot \overline{x_i}\cdot x_i},
\end{equation*}respectively, among $\mathbf{x}\in \mathbb{C}^n\setminus\{0\}$. Moreover, the vectors realizing the minimum (maximum) are the corresponding eigenvectors.
\end{corollary}
\begin{proof}The claim follows directly from Lemma \ref{minmax} for $A$, and it follows from Lemma \ref{minmax} and Proposition \ref{prop:xKx} for $K$. Now, Lemma \ref{minmax} and Proposition \ref{prop:xKx} also imply that the smallest (largest) eigenvalue of $\mathcal{L}$ is the minimizer (maximizer) of
\begin{equation*}
    \RQ_\mathcal{L}(\mathbf{x}):=\frac{\sum_{k=1}^m\biggl|\sum_{j=1}^n\frac{\omega(v_j,e_k)^{-1}}{\sqrt{\deg v_j}}x_j\biggr|^2}{\sum_{i=1}^n \overline{x_i}\cdot x_i},
\end{equation*}with the vectors realizing the minimum (maximum) being the corresponding eigenvectors. This proves the claim for $L$, since $\mathbf{x}$ is an eigenvector for $\mathcal{L}$ with eigenvalue $\lambda$ if and only if $D^{1/2}\mathbf{x}$ is an eigenvector for $L$ with eigenvalue $\lambda$.
\end{proof}

The next lemma \cite[Theorem 6.1.1]{MatrixAnalysis} is often called the Geršgorin disc Theorem.
\begin{lemma}[Geršgorin]\label{lemma:discs}
Let $M=(m_{ij})\in\mathbb{C}^{n\times n}$. The eigenvalues of $M$ lie in the union of Geršgorin discs
\begin{equation*}
    \bigcup_{i=1}^n\biggl\{z\in\mathbb{C}:|z-m_{ii}|\leq \sum_{j\neq i}|m_{ij}|\biggr\}.
\end{equation*}
\end{lemma}
\begin{definition}
The \emph{spectral radius} of $M\in \mathbb{C}^{n\times n}$ is
\begin{equation*}
    \rho(M):=\max\{|\lambda_i|:\lambda_i\text{ is an eigenvalue of } M\}.
\end{equation*}
\end{definition}
\begin{corollary}\label{cor:discs}
Let $M=(m_{ij})\in\mathbb{C}^{n\times n}$. Then,
\begin{equation*}
    \rho(M)\leq \max_{i\in\{1,\ldots,n\}} \biggl( |m_{ii}|+\sum_{j\neq i}|m_{ij}|\biggr).
\end{equation*}
\end{corollary}
\begin{proof}
By Lemma \ref{lemma:discs}, there exists $i\in\{1,\ldots,n\}$ such that
\begin{equation*}
 \rho(M)-|m_{ii}|\leq \sum_{j\neq i}|m_{ij}|.
\end{equation*}That is,
\begin{equation*}
    \rho(M)\leq |m_{ii}|+\sum_{j\neq i}|m_{ij}|.
\end{equation*}The claim follows.
\end{proof}

\subsection{Upper and lower bounds}
Let $\Delta$ and $\nabla$ denote the maximum vertex degree and the maximum edge size of $G$, respectively.
\begin{theorem}\label{thm:rho}
\begin{equation*}
   \rho(A)\leq \Delta(\nabla-1),
\end{equation*}
and this inequality is sharp.
\end{theorem}
\begin{proof} By Corollary \ref{cor:discs},
\begin{equation*}
    \rho(A)\leq \max_{i\in\{1,\ldots,n\}}\sum_{j\neq i}|a_{ij}|=\max_{i\in\{1,\ldots,n\}}\sum_{j\neq i}\sum_{e\in E:v_i,v_j\in e}|\varphi_e(v_i,v_j)|\leq \Delta(\nabla-1).
\end{equation*}To see that the above inequality is sharp, let $G=(V,E,\mathcal{I},\omega)$ be such that:
\begin{itemize}
\item $V=\{v_1,\ldots,v_n\}$ and $E=\{e\}$;
\item $\mathcal{I}=V\times E$;
\item $\omega(v,e)=1$ for all $v\in V$.
\end{itemize}Then, $G$ is a $1$-regular hypergraph and in particular $\Delta=1$, while $\nabla=n$. The adjacency matrix $A=(a_{ij})$ is such that
\begin{equation*}
        a_{ij}=\begin{cases}0 &\text{ if }i=j,\\ -1 &\text{ if }i\neq j.\end{cases}
    \end{equation*}Hence, $A=\textrm{Id}-\mathrm{J}$, where $\mathrm{J}$ is the matrix of ones, and $K=\textrm{Id}-A=\mathrm{J}$. This implies that the eigenvalues of $A$ are $1-n$, with multiplicity $1$, and $1$, with multiplicity $n-1$. In particular, $\rho(A)=n-1=\Delta(\nabla-1)$.\qedhere
\end{proof}

\begin{definition}
The \emph{underlying hypergraph} of $G=(V,E,\mathcal{I},\omega)$ is $G':=(V,E,\mathcal{I},\omega')$, where
\begin{equation*}
    \omega'(v,e):=\begin{cases} 1 & \text{ if } v\in e,\\ 0 & \text{ otherwise.}\end{cases}
\end{equation*}
\end{definition}

For a graph $\Gamma$, the signless Laplacian $Q(\Gamma)=D(\Gamma)+A(\Gamma)$ has received a growing amount of attention.  When finding upper bounds for the Kirkohff Laplacian spectral radius of a graph, it turns out that signless Laplacian can be used since $\lambda_n(K(\Gamma)) \leq \lambda_n(Q(\Gamma))$.  This universal upper bound extends to more general settings of signed graphs \cite{MR1950410} and $\mathbb{T}$-gain graphs \cite{MR2900705}.  More recently, this has also been generalized to the setting of oriented hypergraphs \cite{hyp2014}.  This further generalizes to complex unit hypergraphs, where all of the above structures can be viewed as specializations of the first inequality in the following theorem.

\begin{theorem}
Let $G'$ be the underlying hypergraph of $G$. Then,
\begin{equation*}\label{upperboundK}
    \lambda_n(K(G))\leq \lambda_n(K(G'))\leq \nabla\cdot \Delta,
\end{equation*}and $\lambda_n(K(G'))\leq \nabla\cdot \Delta$ is an equality if and only if $G$ is $\Delta$-regular and $\nabla$-uniform. Similarly,
\begin{equation*}\label{upperboundL}
    \lambda_n(L(G))\leq \lambda_n(L(G'))\leq \nabla,
\end{equation*}and $\lambda_n(L(G'))\leq \nabla$ is an equality if and only if $G$ is $\nabla$-uniform. 
\end{theorem}
\begin{proof}
Let $\mathbf{x}=(x_1,\ldots,x_n)\in\mathbb{C}^n$ be a unit eigenvector of $K(G)$ with corresponding eigenvalue $\lambda_n(K(G))$. Then, by Corollary \ref{cor:rq},
\begin{equation*}
    \lambda_n(K(G))=\sum_{e\in E}\biggl(\sum_{v_i\in e}\omega(v_i,e)^{-1}x_i\biggr)^2
    \leq \sum_{e\in E}\Biggl(\sum_{v_i\in e}|x_i|\Biggr)^2\leq \lambda_n(K(G')).
\end{equation*}Similarly, let $\mathbf{y}=(y_1,\ldots,y_n)\in\mathbb{C}^n$ be a unit eigenvector of $K(G')$ with corresponding eigenvalue $\lambda_n(K(G'))$. Then
\begin{equation}\label{eq:G'}
     \lambda_n(K(G'))=\sum_{e\in E}\biggl(\sum_{v_i\in e}y_i\biggr)^2
    \leq \sum_{e\in E}\Biggl(\sum_{v_i\in e}|y_i|\Biggr)^2\leq \lambda_n(K(G')),
\end{equation}therefore all inequalities in \eqref{eq:G'} are equalities. Now, for each $e\in E$,
\begin{align*}
    \Biggl(\sum_{v_i\in e}|y_i|\Biggr)^2&=\sum_{v_i\in e} y_i^2+\sum_{\{i,j\}:v_i\neq v_j\in e}2\cdot |y_i|\cdot |y_j|\\
    &\leq \sum_{v_i\in e} y_i^2+\sum_{\{i,j\}:v_i\neq v_j\in e}\left(y_i^2+ y_j^2\right)\\
    &= \sum_{v_i\in e}y_i^2+\sum_{v_i\in e}(|e|-1) \cdot y_i^2\\
    &=|e|\cdot \sum_{v_i\in e} y_i^2,
\end{align*}with equality if and only if $y_i$ is constant for all $v_i\in e$. Hence,
\begin{align*}
    \lambda_n(K(G'))&\leq \sum_{e\in E}|e|\cdot \Biggl(\sum_{v_i\in e} y_i^2\Biggr)\\ &\leq \nabla \cdot \Biggl(\sum_{e\in E,v_i\in e}y_i^2\Biggr)\\ &=\nabla\cdot \Biggl(\sum_{v_i\in V}\deg v_i\cdot y_i^2\Biggr)\\
    &\leq \nabla \cdot \Delta \cdot \Biggl(\sum_{v_i\in V}y_i^2\Biggr)\\
    &=\nabla \cdot \Delta,
\end{align*}with equality if and only if $\deg v_i=\Delta$ is constant for all $v_i\in V$, $|e|=\nabla$ is constant for all $e\in E$ and, for all $e$, $y_i$ is constant for all $v_i\in e$. This proves the claim for $K$. Now, proving that
\begin{equation*}
    \lambda_n(L(G))\leq \lambda_n(L(G')),
\end{equation*}can be done in the same way as for $K$, by normalizing $\mathbf{x}=(x_1,\ldots,x_n)\in\mathbb{C}^n$ so that $\sum_{i=1}^n\deg v_i\cdot x_i=1$. The fact that $\lambda_n(L(G'))\leq \nabla$, with equality if and only if $G$ is $\nabla$-uniform, is proved in \cite{Sharp}.
\end{proof}
Given $S\subseteq V$, let $\vol S:=\sum_{v\in S}\deg v$.
\begin{theorem}\label{thm:S}
If $S \subseteq V$ is such that, for each $v \in S$, $\omega(v):=\omega(v,e)$ is constant for all $e\in E$ with $v\in e$, then
\begin{equation*}
    \lambda_1(K)\leq \frac{\sum_{e\in E}|e\cap S|^2}{|S|}\leq \lambda_n(K),
\end{equation*}and
\begin{equation*}
    \lambda_1(L)\leq \frac{\sum_{e\in E}|e\cap S|^2}{\vol S}\leq \lambda_n(L).
\end{equation*}
\end{theorem}
\begin{proof}Let $S\subseteq V$ be such that, for each $v \in S$, $\omega(v):=\omega(v,e)$ is constant for all $e\in E$ with $v\in e$. Let $\mathbf{x}=(x_1,\ldots,x_n)\in\mathbb{C}^n$ be defined by
\begin{equation*}
    x_i:=\begin{cases}\omega(v_i) &\text{ if }v_i\in S,\\
    0 & \text{ otherwise.}
    \end{cases}
\end{equation*}Then, by Corollary \ref{cor:rq},

\begin{equation*}
   \lambda_1(K)\leq \RQ_K(\mathbf{x})=\frac{\sum_{e\in E}\bigl|\sum_{v_j\in e}\omega(v_j,e)^{-1}x_j\bigr|^2}{\sum_{i=1}^n \overline{x_i}\cdot x_i}=\frac{\sum_{e\in E}|e\cap S|^2}{|S|}\leq \lambda_n(K),
\end{equation*}and
\begin{equation*}
    \lambda_1(L)\leq \RQ_L(\mathbf{x})=\frac{\sum_{e\in E}\bigl|\sum_{v_j\in e}\omega(v_j,e)^{-1}x_j\bigr|^2}{\sum_{i=1}^n \deg v_i\cdot \overline{x_i}\cdot x_i}= \frac{\sum_{e\in E}|e\cap S|^2}{\vol S}\leq \lambda_n(L). \qedhere
\end{equation*}

\end{proof}

\begin{definition}
A subset $S\subseteq V$ is \emph{independent} if $\#(S\cap e)\le 1$ for all $e\in E$. The \emph{independence number} of $\Gamma$ is
\begin{equation*}
    \alpha:=\max\{|S|:S\subseteq V \text{ independent}\}.
\end{equation*}
\end{definition}
\begin{corollary}\label{cor:independent}
If $S\subseteq V$ is independent,
\begin{equation*}
    \frac{\vol S}{|S|}\leq \lambda_n(K).
\end{equation*}
\end{corollary}

\begin{proof}
Since $S$ is an independent set, by Lemma \ref{lemma:switching} we can assume, up to edge-switching, that given $v\in S$,
\begin{equation*}
    \omega(v):=\omega(v,e),
\end{equation*}
is constant for all $e\in E$ with $v\in e$. Therefore, by Theorem \ref{thm:S},
\begin{equation*}
    \lambda_1(K)\leq \frac{\sum_{e\in E}|e\cap S|^2}{|S|}=\frac{\vol S}{|S|} \leq \lambda_n(K).\qedhere
\end{equation*}
\end{proof}

\begin{remark}The proof used in Corollary \ref{cor:independent} in the case of $L$ would simply imply that
\begin{equation*}
    \lambda_1(L)\leq 1\leq \lambda_n(L), 
\end{equation*}but this statement is trivial since we know that $\sum_{i=1}^n\lambda_i(L)=n$, therefore the average of the eigenvalues is $1$. However, the eigenvalues of $L$ and $A$ also relate to the independence sets. In fact, with the same proof as the one in \cite[Theorem 3.4]{Independence}, one can see that
\begin{equation*}
       \alpha\leq \min\{|\{i:\lambda_i(L)\leq 1\}|,|\{i:\lambda_i(L)\geq 1\}|\},
  \end{equation*}and similarly
  \begin{equation*}
       \alpha\leq \min\{|\{i:\lambda_i(A)\leq 0\}|,|\{i:\lambda_i(A)\geq 0\}|\}.
  \end{equation*}
  \end{remark}

\begin{corollary}
\begin{equation*}
    \max\{\Delta,\nabla\}\leq \lambda_n(K).
\end{equation*}
\end{corollary}
\begin{proof}
Let $v\in V$ with $\deg v=\Delta$. Then, $S=\{v\}$ is an independent set with $\vol S=\Delta$ and $|S|=1$. By Corollary \ref{cor:independent}, 
\begin{equation*}
    \Delta \leq \lambda_n(K).
\end{equation*}By taking the dual hypergraph, this also implies that
\begin{equation*}
    \nabla \leq \lambda_m(K(G^*))=\lambda_m(K^*(G))=\lambda_n(K(G)).
\end{equation*}Hence,
\begin{equation*}
    \max\{\Delta,\nabla\}\leq \lambda_n(K). \qedhere
\end{equation*} 
\end{proof}

\bibliographystyle{unsrt}
\bibliography{ComplexUnit}
				\end{document}